\documentclass{article}

\usepackage{layout}
\usepackage{amssymb,amsthm,amsmath,enumerate,graphicx,color}
\usepackage[colorlinks,citecolor=blue,linkcolor=red,urlcolor=blue]{hyperref}

\theoremstyle{plain}
\newtheorem{theorem}{Theorem}
\newtheorem{proposition}[theorem]{Proposition}
\newtheorem{lemma}[theorem]{Lemma}

\newtheorem{remark}[theorem]{Remark}
\newtheorem{example}[theorem]{Example}

\makeatletter
\def\section{\@startsection {section}{1}{\z@}{-3.5ex plus -1ex minus 
-.2ex}{2.3ex plus .2ex}{\large\bf}}
\makeatother

\makeatletter
\def\subsection{\@startsection {subsection}{1}{\z@}{-3.5ex plus -1ex minus 
-.2ex}{2.3ex plus .2ex}{\large\it}}
\makeatother

\begin{document}

\title{\textbf{\Large A strong and weak approximation scheme for\\ stochastic differential equations driven by\\ a time-changed Brownian motion}}

\author{Ernest Jum\thanks{Winston--Salem, NC, USA. Email: ernyjum@gmail.com} \and Kei Kobayashi\thanks{Department of Mathematics, The University of Tennessee, 1403 Circle Drive, Knoxville, TN 37996, USA. Email: kkobayas@utk.edu}}
\date{ }
\maketitle                                                       
\vspace{-2mm}
\renewcommand{\thefootnote}{\fnsymbol{footnote}}

\begin{abstract}
This paper establishes a discretization scheme for a large class of stochastic differential equations driven by a time-changed Brownian motion with drift, where the time change is given by a general inverse subordinator. The scheme involves two types of errors: one generated by application of the Euler--Maruyama scheme and the other ascribed to simulation of the inverse subordinator. 
With the two errors carefully examined, the orders of strong and weak convergence are established. In particular, an improved error estimate for the Euler--Maruyama scheme is derived, which is required to guarantee the strong convergence. Numerical examples are attached to support the convergence results.
\end{abstract}

\large

\noindent
\textit{}

\section{Introduction}\label{section_introduction}

Time-fractional versions of classical Kolmogorov or Fokker--Planck equations have been widely used to study dynamics of anomalous diffusions observed in e.g.\ physics \cite{MetzlerKlafter00,Zaslavsky}, finance \cite{GMSR,Magdziarz_BS}, hydrology \cite{BWM}, and cell biology \cite{Saxton}. 
Such fractional partial differential equations are known to be connected with limit processes arising from certain weakly
convergent sequences or triangular arrays of continuous-time random walks. 
These limit processes are time-changed L\'evy processes, where the time changes
are given by the inverses of certain subordinators (see \cite{MS_1,MS_2} for details).

In \cite{HKU-1}, the authors identify a wide class of stochastic differential
equations (SDEs) whose associated Kolmogorov-type equations are
 time-fractional distributed order pseudo-differential equations, where
the driving processes of the SDEs are time-changed L\'evy processes. In connection with these SDEs, a detailed discussion of stochastic integrals and SDEs driven by time-changed semimartingales is provided in \cite{Kobayashi}.
A recent work \cite{ScalasViles}
employs a continuous-time random walk approach presented by \cite{Burr} to construct sequences which converge weakly to stochastic integrals driven by time-changed stable L\'evy processes, where the time change is given by the inverse of a stable subordinator.

In this paper, combining the duality principle established in \cite{Kobayashi} (see Lemma \ref{Lemma_duality} in Section \ref{section_preliminaries}) 
with an idea of approximations of inverse subordinators described in \cite{Magdziarz_simulation,Magdziarz_spa},
we will present a discretization scheme for a large class of SDEs driven by a time-changed Brownian motion which are of the form 
\begin{align*}
	Y(t)=y_0+\int_0^t b(E(r),Y(r)) dE(r)+\int_0^t \sigma(E(r),Y(r)) dB(E(r)),
\end{align*}
where $B$ is a Brownian motion and $E$ is an independent time change given by an inverse subordinator with infinite L\'evy measure (to be precisely defined in Section \ref{section_preliminaries}). 
Our approximation scheme extends a scheme presented in Section III of \cite{GajdaMagdziarz} to SDEs of the above form with general time-dependent coefficients and time changes; in that paper, the coefficients are $b(t,x)\equiv b(x)$ and $\sigma(t,x)\equiv 1$ and the time change $E$ is the inverse of an exponentially tempered stable subordinator.  Moreover, we will establish 
both strong and weak convergence of our approximation process to the exact solution of the above SDE with the respective orders of convergence specified, which is not investigated in \cite{GajdaMagdziarz} and hence serves as the main contribution of this paper.

The rest of the paper is organized as follows. Section \ref{section_preliminaries} precisely defines the class of SDEs to be considered in this paper and provides preliminary facts concerning such SDEs. Section \ref{section_scheme} establishes the main results of this paper; i.e.\ strong and weak convergence of our approximation scheme along with their respective orders.  Discussions are given with emphasis on analysis of two types of errors: one generated by the Euler--Maruyama scheme and the other ascribed to the approximation of the inverse subordinator. 
In particular, in Proposition \ref{Proposition_A(T)}, we derive an error estimate concerning strong convergence of the Euler--Maruyama scheme using a technique significantly different from the well-known method appearing in \cite{KloedenPlaten}. Namely, we utilize Burkholder's inequality to obtain a sharper error estimate, which is essential for the derivation of Theorems \ref{Theorem_main2} and \ref{Theorem_main3}; see item 3) of Remark \ref{Remark_main} for details on this issue. Section \ref{section_simulation} provides numerical examples that support the convergence results.

\section{Preliminaries}\label{section_preliminaries}

This section provides necessary backgrounds for SDEs driven by a time-changed Brownian motion. 
Throughout the paper, a complete probability space $(\Omega,\mathcal{F},\mathbb{P})$ with a filtration $(\mathcal{F}_t)_{t\ge 0}$ satisfying the usual conditions is fixed.

Let $D$ be an $(\mathcal{F}_t)$-adapted subordinator with Laplace exponent $\psi$ and L\'evy measure $\nu$; i.e.\ $D$ is a one-dimensional nondecreasing L\'evy process with c\`adl\`ag paths starting at 0 and Laplace transform 
\begin{align}\label{def_LaplaceExponent}
	\mathbb{E}[e^{-sD(t)}]=e^{-t\psi(s)}, \ \ \textrm{where} \ \ 
	\psi(s)=a s +\int_0^\infty (1-e^{-sx}) \nu(dx), \ \ s> 0,
\end{align}
with 
$a\ge 0$ and
$\int_0^\infty (x\wedge 1)  \nu(dx)<\infty$.
We assume that the L\'evy measure is infinite, i.e.\ $\nu(0,\infty)=\infty$, which implies that $D$ has strictly increasing paths with infinitely many jumps (see e.g.\ Theorem 21.3 of \cite{Sato}). 
Let $E$ be the inverse of $D$;
\begin{align}\label{inverse}
	E(t):=\inf\{u>0; D(u)>t\}, \ \ t\ge 0. 
\end{align}
Since $D$ has strictly increasing paths, the process $E$, called an \textit{inverse subordinator}, has continuous, nondecreasing paths. 
Moreover, $E$ is a continuous $(\mathcal{F}_t)$-time change (see e.g.\ Lemma 2.7 of \cite{Kobayashi}) and hence the time-changed filtration $(\mathcal{F}_{E(t)})_{t\ge 0}$ is well-defined.

Let $B$ be an $m$-dimensional $(\mathcal{F}_t)$-adapted Brownian motion starting at 0. The \textit{time-changed Brownian motion} $B\circ E$ is widely used to model subdiffusions, where particles spread more slowly than the classical Brownian motion particles do. In particular, the particles represented by $B\circ E$ are trapped and immobile during the constant periods of $E$.
Consider the SDE
\begin{align}\label{SDE_new}
	Y(t)=y_0+\int_0^t b(E(r),Y(r)) dE(r)+\int_0^t \sigma(E(r),Y(r)) dB(E(r)),
\end{align}
where $y_0\in\mathbb{R}^d$ is a non-random constant, and $b(t,x):[0,\infty)\times \mathbb{R}^d\to \mathbb{R}^d$ and $\sigma(t,x):[0,\infty)\times \mathbb{R}^d\to \mathbb{R}^{d\times m}$ are measurable functions for which there is a positive constant $K$ such that
\begin{align}
	\label{SDE_condition1} &|b(t,x)-b(t,y)|+|\sigma(t,x)-\sigma(t,y)|\le K|x-y|,\\ 
	\label{SDE_condition2} &|b(t,x)|+|\sigma(t,x)|\le K(1+|x|),\\
	\label{SDE_condition3} &|b(s,x)-b(t,x)|+|\sigma(s,x)-\sigma(t,x)|\le K(1+|x|)|s-t|^{\gamma} 
\end{align}  
for all $x,y\in\mathbb{R}^d$ and $s,t\ge 0$, where $\gamma$ is a fixed positive constant and $|\cdot|$ denotes the Euclidean norms of appropriate dimensions. 
 Here, the stochastic integrals appearing in SDE \eqref{SDE_new} are understood within the framework of stochastic integrals driven by semimartingales as the integrators $E$ and $B\circ E$ are both $(\mathcal{F}_{E(t)})$-semimartingales due to Corollary 10.12 of \cite{Jacod}. 
The initial value $y_0$ is taken to be a non-random constant only for simplicity of discussions; all the results appearing in this paper can be easily generalized with a random initial value satisfying appropriate conditions such as existence of moments.
 
Note that under conditions \eqref{SDE_condition1} and \eqref{SDE_condition2}, SDE \eqref{SDE_new} has a unique strong solution $Y$ on $[0,\infty)$ (see Theorem 7 in Chapter V of \cite{Protter}; also see Lemma 4.1 of \cite{Kobayashi}). 
The Kolmogorov-type equation associated with the solution $Y$ is known. In particular, if $D$ is a $\beta$-stable subordinator independent of the Brownian motion $B$, and if $b$ and $\sigma$ are autonomous coefficients satisfying some regularity conditions, then the function $u(t,x):=\mathbb{E}[\varphi(Y(t))|Y(0)=x]$, where $\varphi\in C^2_0(\mathbb{R}^d)$, satisfies the time-fractional Kolmogorov-type equation 
\[
	\partial_t^\beta u(t,x)=\mathcal{A}u(t,x)
\]
 with $\partial_t^\beta$ being the Caputo fractional derivative of order $\beta$ and $\mathcal{A}=b(x)\partial_x+\frac 12 \sigma^2(x)\partial_x^2$; see \cite{HKU-1} for this special case. General cases are treated in the recent papers \cite{MagdziarzSchilling,MagdziarzZorawik}.

Condition \eqref{SDE_condition3} necessarily holds for autonomous coefficients and is needed to obtain Proposition \ref{Proposition_A(T)},
which will be used to derive the main results in Theorems \ref{Theorem_main2} and \ref{Theorem_main3}.
It is also worth noting that $E$ and $B\circ E$ are non-Markovian and do not have independent or stationary increments (see \cite{MS_1}), which makes it difficult to simulate sample paths of the solution $Y$ to SDE \eqref{SDE_new} via direct applications of well-known approximation schemes such as the Euler--Maruyama scheme.

The duality principle in \cite{Kobayashi} reveals a deep connection between SDE \eqref{SDE_new} and the classical It\^o SDE 
\begin{align}\label{SDE_classical}
	X(t)=y_0+\int_0^t b(r,X(r)) dr+\int_0^t \sigma(r,X(r)) dB(r).  
\end{align}

\begin{lemma}[Duality principle {\cite[Theorem 4.2]{Kobayashi}}]\label{Lemma_duality}
Suppose that $b(t,x)$ and $\sigma(t,x)$ satisfy conditions \eqref{SDE_condition1} and \eqref{SDE_condition2}. 
If $Y$ is the unique solution to SDE \eqref{SDE_new}, then the time-changed process $X:=Y\circ D$ is an $(\mathcal{F}_{t})$-semimartingale which is the unique solution to SDE \eqref{SDE_classical}. 
On the other hand, if $X$ is the unique solution to SDE \eqref{SDE_classical}, then the time-changed process $Y:=X\circ E$ is an $(\mathcal{F}_{E(t)})$-semimartingale which is the unique solution to SDE \eqref{SDE_new}.  
\end{lemma}

Note that the continuity of the sample paths of $E$ is necessary for the duality principle to hold 
(see Example 2.5 of \cite{Kobayashi}). 
Therefore, the results to be presented in this paper cannot be immediately extended to the case where the L\'evy measure of $D$ is finite (in which case the inverse $E$ has jumps and the duality principle no longer holds).

\section{An approximation scheme and pertinent results on convergence}\label{section_scheme}

Throughout the paper, we assume that the Brownian motion $B$ is independent of the subordinator $D$. Our discretization scheme for the solution $Y$ to SDE \eqref{SDE_new} on a fixed interval $[0,T]$ is two-fold --- to apply the Euler--Maruyama scheme to SDE \eqref{SDE_classical} to construct a process $X_\delta$ approximating the solution $X$ (see \eqref{def_Xdelta}--\eqref{def_interpolation} below), and to approximate the inverse subordinator $E$ by a process $E_\delta$ to be defined in \eqref{def_Spsidelta} (which was introduced in \cite{Magdziarz_simulation,Magdziarz_spa}). 
Here, $\delta\in(0,1)$ denotes the equidistant step size to be taken in the discretization scheme.
The duality principle (Lemma \ref{Lemma_duality}) suggests the use of
the composition $Y_\delta:=X_\delta\circ E_\delta$ as a process approximating the solution $Y$ of SDE \eqref{SDE_new}. However, to guarantee the reliability of our approximation scheme, we must carefully analyze two different errors: one generated by the Euler--Maruyama scheme and the other due to the approximation of the inverse subordinator. The first part of the this section is devoted to discussions of these errors.

\subsection{Improved error estimates for the Euler--Maruyama scheme}

In this subsection, we derive important error estimates concerning the Euler--Maruyama scheme; see Propositions \ref{Proposition_A(T)} and \ref{Proposition_weak2}. These estimates improve those given in Theorems 10.2.2 and 14.5.1 (with $\beta=1$) of \cite{KloedenPlaten}. In particular, a method to be used to derive Proposition \ref{Proposition_A(T)} is significantly different from the one employed in 
\cite{KloedenPlaten}. To obtain the improved error bound, we will utilize Burkholder's inequality.

For a fixed $\delta\in(0,1)$,
we apply the Euler--Maruyama scheme to SDE \eqref{SDE_classical} on the positive real line $[0,\infty)$ by choosing discretization times  
$\tau_n:=n\delta$, $n=0,1,2,\ldots$, with equal step size $\delta$, and then setting 
\begin{align}\label{def_Xdelta}
	X_\delta(0):= y_0, \ \ 
 	X_\delta(\tau_{n+1})
	:={ }&X_\delta(\tau_n)+b(\tau_n,X_\delta(\tau_n))(\tau_{n+1}-\tau_n)\\
		&+\sigma(\tau_n,X_\delta(\tau_n))(B(\tau_{n+1})-B(\tau_n))\notag
 \end{align}
 for $n=0,1,2,\ldots$. A continuous-time process $X_\delta=(X_\delta(t))_{t\ge 0}$ is defined by continuously interpolating the discrete-time process $(X_\delta(\tau_n))_{n=0,1,2,\ldots}$ by 
 \begin{align}\label{def_interpolation}
	X_\delta(t)
	:={}&X_\delta(\tau_n)+b(\tau_n,X_\delta(\tau_n))(t-\tau_n)\\
		&+\sigma(\tau_n,X_\delta(\tau_n))(B(t)-B(\tau_n))
	 \ \textrm{whenever} \ t\in[\tau_n,\tau_{n+1}].\notag
 \end{align}
 The interpolation is for a theoretical purpose only and 
the information of the interpolated values is not used for simulation of sample paths of the solution $Y$ of SDE \eqref{SDE_new} (see Section \ref{section_simulation} for details).  

It is known that the Euler approximation with $\gamma=1$ in condition \eqref{SDE_condition3} has the order of (uniform) strong convergence $0.5$. The exact statement is provided in the following lemma, which appears in \cite{KloedenPlaten}.

\begin{lemma}[{\cite[Theorem 10.2.2, Remark 10.2.3]{KloedenPlaten}}]\label{Lemma_Euler}
Let $X$ be the solution to SDE \eqref{SDE_classical} on a bounded interval $[0,T_\ast]$ satisfying conditions \eqref{SDE_condition1}, \eqref{SDE_condition2} and \eqref{SDE_condition3} with $\gamma=1$. 
For a fixed $\delta\in(0,1)$, let $X_\delta$ be the process defined in \eqref{def_Xdelta}--\eqref{def_interpolation} on $[0,T_\ast]$.
Then there exists a positive constant $A$ not depending on $\delta$ 
such that
\begin{align}\label{Estimate_Euler}
	\mathbb{E}\biggl[\sup_{0\le s\le t}|X(s)-X_\delta(s)|\biggr]
	\le A\delta^{1/2} \ \ \textrm{for all} \ \ t\in [0,T_\ast].
\end{align}
\end{lemma}

Note that we must assume condition \eqref{SDE_condition3} with $\gamma=1$ here, which is not needed to simply guarantee the existence of a unique strong solution $X$ to SDE \eqref{SDE_classical}. The proof of this lemma provided in \cite{KloedenPlaten} allows the constant $A$ in \eqref{Estimate_Euler} to depend on the time horizon $T_\ast$. However, to obtain the main results of this paper, we need to refine the above statement in such a way that the processes $X$ and $X_\delta$ are defined on the positive real line $[0,\infty)$ (rather than on any bounded interval $[0,T_\ast]$) and that $A$ in \eqref{Estimate_Euler} depends on $t$ (rather than on any fixed time horizon $T_\ast$).
More precisely, the following improved version of Lemma \ref{Lemma_Euler} will be required.

\begin{proposition}\label{Proposition_A(T)}
Let $X$ be the solution to SDE \eqref{SDE_classical} on the positive real line $[0,\infty)$ satisfying conditions \eqref{SDE_condition1}, \eqref{SDE_condition2} and \eqref{SDE_condition3}. 
For a fixed $\delta\in(0,1)$, let $X_\delta$ be the process on $[0,\infty)$ defined in \eqref{def_Xdelta}--\eqref{def_interpolation}. 
Then there exists a positive constant $C$ not depending on $\delta$ or $t$ such that
\begin{align}\label{Estimate_Euler2}
	\mathbb{E}\biggl[\sup_{0\le s\le t}|X(s)-X_\delta(s)|^2\biggr] 
	\le Ce^{Ct}\delta^{\min(2\gamma,1)} \ \ \textrm{for all} \ \ t\ge 0.
\end{align}
\end{proposition}

To prove Proposition \ref{Proposition_A(T)}, we will need the following simple lemma, which will be employed to derive Theorem \ref{Theorem_main2} as well.

\begin{lemma}\label{Lemma_modulus}
Let $X$ be the solution to SDE \eqref{SDE_classical} on $[0,\infty)$ satisfying conditions \eqref{SDE_condition1} and \eqref{SDE_condition2}.
Then for any $\delta\in(0,1)$ and any two time points $s$ and $t$ with $0\le t-s\le \delta$, the inequality 
\[
	\mathbb{E}[|X(t)-X(s)|^2]\le \delta Ce^{Ct}
\]
 holds, 
where $C$ is a constant not depending on $\delta$ or $t$.
\end{lemma}

\begin{proof}
For notational simplicity, we give a proof only in the case when $d=m=1$; a multidimensional generalization is straightforward. 
For $s$ and $t$ such that $0\le t-s\le \delta (<1)$, it follows from the integral representation \eqref{SDE_classical}, the inequality $(x+y)^2\le 2x^2+2y^2$, and the Cauchy--Schwartz inequality that
\begin{align*}
	\mathbb{E}[|X(t)-X(s)|^2]
	&\le 2\mathbb{E}\biggl[(t-s)\int_s^t |b(r,X(r))|^2 dr\biggr]
		+2\mathbb{E}\biggl[\int_s^t |\sigma(r,X(r))|^2 dr\biggr],\notag
\end{align*}
which is dominated by $8K^2\int_s^t (1+\mathbb{E}[|X(r)|^2])  dr$ due to condition \eqref{SDE_condition2}. 
Since the initial value for SDE \eqref{SDE_classical} is assumed non-random, for each $\ell\in \mathbb{N}$, there exists a constant $C$ depending on $\ell$ and $K$ but not on $r$ such that
\begin{align}\label{moment_estimate}
	\mathbb{E}[|X(r)|^{2\ell}]\le Ce^{Cr} \ \ \textrm{for all} \ \ r\ge 0;
\end{align}
see e.g.\ Theorem 4.5.4 of \cite{KloedenPlaten}. Using this estimate with $\ell=1$, we obtain 
\[
	\mathbb{E}[|X(t)-X(s)|^2]
	\le 8K^2(1+Ce^{C t})\delta, 
\]
which completes the proof. 
\end{proof}

\begin{proof}[Proof of Proposition {\ref{Proposition_A(T)}}]
Again, for notational simplicity, we provide a proof only in the case when $d=m=1$. 
The definition of the process $X_\delta$ in \eqref{def_Xdelta}--\eqref{def_interpolation} implies that 
\begin{align}\label{def_Xdelta_2}
	X_\delta(t)=y_0+\int_0^t b(\tau_{n_r},X_\delta(\tau_{n_r})) dr+\int_0^t \sigma(\tau_{n_r},X_\delta(\tau_{n_r})) dB(r), \ \ t\ge 0,
\end{align}
where $n_r:=\max\{n=\{0\}\cup \mathbb{N}; \tau_n\le r\}$ so that $\tau_{n_r}=\tau_n$ whenever $r\in[\tau_n,\tau_{n+1}]$.  Let $\tilde{X}(t):=X(t)-X_\delta(t)$ and let $[\tilde{X},\tilde{X}]$ denote the quadratic variation process of $\tilde{X}$ (see e.g.\ \cite{Protter}). Since 
\[
	\tilde{X}^2(s)=2\int_0^s \tilde{X}(r) d\tilde{X}(r)+[\tilde{X},\tilde{X}](r),
\]
 the integral  representations \eqref{SDE_classical} and \eqref{def_Xdelta_2} yield  
\begin{align*}
	\tilde{X}^2(s)
	={ }&2\int_0^s \tilde{X}(r)[b(r,X(r)-b(\tau_{n_r},X_\delta(\tau_{n_r}))] dr\\
	&+\int_0^s \tilde{X}(r)[\sigma(r,X(r))-\sigma(\tau_{n_r},X_\delta(\tau_{n_r}))] dB(r)\\
	&+\int_0^s [\sigma(r,X(r))-\sigma(\tau_{n_r},X_\delta(\tau_{n_r}))]^2 dr.
\end{align*}
To deal with the three terms separately, write 
\[
	Z(t):=\mathbb{E}[\sup_{0\le s\le t}\tilde{X}^2(s)]=I_1(t)+I_2(t)+I_3(t),
\]
with $I_i(t)$ denoting the expectation of the supremum over $s\in[0,t]$ of the $i$th term. 

As for the term $I_3(t)$, first note that conditions \eqref{SDE_condition1} and \eqref{SDE_condition3} and the trivial fact that $0\le r-\tau_{n_r}\le \delta$ imply that
\begin{align}\label{estimate_sigma}
	&|\sigma(r,X(r))-\sigma(\tau_{n_r},X_\delta(\tau_{n_r}))|\\
	&\le |\sigma(r,X(r))-\sigma(\tau_{n_r},X(r))|
		+|\sigma(\tau_{n_r},X(r))-\sigma(\tau_{n_r},X(\tau_{n_r}))|\notag\\
	&\ \ \ +|\sigma(\tau_{n_r},X(\tau_{n_r}))-\sigma(\tau_{n_r},X_\delta(\tau_{n_r}))|\notag\\
	&\le K(1+|X(r)|)\delta^{\gamma}+K|X(r)-X(\tau_{n_r})|+K|\tilde{X}(\tau_{n_r})|.\notag
\end{align}
Hence, by the inequality $(x+y+z)^2\le 3(x^2+y^2+z^2)$, 
the estimate \eqref{moment_estimate} with $\ell=1$,
Lemma \ref{Lemma_modulus}, and the fact that $0\le r-\tau_{n_r}\le \delta$, it follows that
\begin{align*}
	I_3(t)
	&= \mathbb{E}\biggl[\sup_{0\le s\le t}
		\int_0^s [\sigma(r,X(r))-\sigma(\tau_{n_r},X_\delta(\tau_{n_r}))]^2 dr\biggr]\\
	&\le 3K^2 \mathbb{E}\biggl[
		\int_0^t \bigl\{2\delta^{2\gamma}(1+X^2(r))+|X(r)-X(\tau_{n_r})|^2+\tilde{X}^2(\tau_{n_r})\bigr\} dr\biggr]\\
	&\le 3K^2 \int_0^t \bigl\{2\delta^{2\gamma}(1+ C e^{Cr})+\delta C e^{Cr}+Z(r)\bigl\} dr,
\end{align*}
where $C$ represents a generic positive constant depending only on $K$ and $y_0$ (but not on $t$ or $\delta$), \textit{the value of which may change from line to line throughout the proof}. Hence, using the fact that any polynomials in $t$ are dominated above by functions of the form $Ce^{Ct}$, we obtain $I_3(t)\le \delta^{\min(2\gamma,1)} Ce^{Ct}+C\int_0^t Z(r)  dr$.

To derive an estimate for the term $I_1(t)$, use the trivial inequality $2xy\le x^2+y^2$ to observe that 
\begin{align*}
	I_1(t)
	&\le \mathbb{E}\biggl[
		\int_0^t 2|\tilde{X}(r)||b(r,X(r)-b(\tau_{n_r},X_\delta(\tau_{n_r}))| dr
\biggr]\\
	&\le \mathbb{E}\biggl[\int_0^t \bigl\{\tilde{X}^2(r)+[b(r,X(r)-b(\tau_{n_r},X_\delta(\tau_{n_r}))]^2\bigl\} dr\biggr]\\
	&\le \int_0^t Z(r) dr+\mathbb{E}\biggl[\int_0^t [b(r,X(r)-b(\tau_{n_r},X_\delta(\tau_{n_r}))]^2 dr\biggr].
\end{align*}
An upper bound for the second term of the last line is easily obtained in a manner similar to the estimation of $I_3(t)$ above; consequently, an estimate of the form  $I_1(t)\le \delta^{\min(2\gamma,1)} Ce^{Ct}+C\int_0^t Z(r)  dr$ again follows.

The term $I_2(t)$ can be estimated with the help of Burkholder's inequality (see e.g.\ Theorem 48 of Chapter IV of \cite{Protter}) as
\begin{align*}
	I_2(t)
	&\le \mathbb{E}\biggl[\sup_{0\le s\le t}
		\biggl|\int_0^s \tilde{X}(r)[\sigma(r,X(r)-\sigma(\tau_{n_r},X_\delta(\tau_{n_r}))] dB(r)\biggr|\biggr]\\
	&\le C \mathbb{E}\biggl[
		\biggl(\int_0^t\tilde{X}^2(r)[\sigma(r,X(r)-\sigma(\tau_{n_r},X_\delta(\tau_{n_r}))]^2 dr\biggr)^{1/2}\biggr]\\
	&\le C \mathbb{E}\biggl[
		\biggl(\sup_{0\le s\le t}\tilde{X}^2(s)\int_0^t[\sigma(r,X(r)-\sigma(\tau_{n_r},X_\delta(\tau_{n_r}))]^2 dr\biggr)^{1/2}\biggr]\\
	&\le \dfrac 12 Z(t)+2C^2\mathbb{E}\biggl[\int_0^t[\sigma(r,X(r)-\sigma(\tau_{n_r},X_\delta(\tau_{n_r}))]^2 dr\biggr],
\end{align*}
where we used the inequality $(xy)^{1/2}\le x/(2C)+2Cy$ for $x,y\ge 0$. This upper bound for $I_2(t)$ and the estimation of $I_3(t)$ above together imply that $I_2(t)\le Z(t)/2+\delta^{\min(2\gamma,1)} Ce^{Ct}+C\int_0^t Z(r)  dr$.

Now, the three estimated terms combined, we obtain 
\begin{align}\label{estimate_Gronwall}
	Z(t)\le \delta^{\min(2\gamma,1)} Ce^{Ct}+C\int_0^t Z(r)  dr, \ \ t\ge 0.
\end{align}
This, together with the classical result on Gronwall's inequality (see e.g.\ Theorem 2 on p.353 of \cite{Mitrinovic}), yields an  estimate of the form \eqref{Estimate_Euler2}. 
\end{proof}

\begin{remark}\label{Remark_Euler}
\begin{em}
1) Although the statement of Theorem 10.2.2 of \cite{KloedenPlaten} (equivalent to Lemma  \ref{Lemma_Euler} in this paper) requires condition \eqref{SDE_condition3} (with $\gamma=1$) for \textit{all} $s,t\ge 0$, 
the estimate in \eqref{estimate_sigma} shows that, in order for Proposition \ref{Proposition_A(T)} to hold for all $\delta\in(0,\delta_0)$ for a fixed $\delta_0\in (0,1]$, it is sufficient to assume condition \eqref{SDE_condition3} 
for $s,t\ge 0$ satisfying $|s-t|\le \delta_0$.  
This observation extends the class of coefficients of SDE \eqref{SDE_new} to which our approximation scheme applies.

2) Proposition \ref{Proposition_A(T)} is \textit{not} a consequence of a simple modification of the proof of Lemma \ref{Lemma_Euler} found in \cite{KloedenPlaten}; if we followed that proof, then the upper bound in the estimate \eqref{Estimate_Euler2} would take the form 
$Ce^{C(t^2+t)}\delta^{\min(2\gamma,1)}$,
which clearly grows faster as $t\to\infty$ than the bound we derived in the above proof. 
We would obtain such a rough estimate since
the inequality \eqref{estimate_Gronwall} would be replaced by
\[
	Z(t)\le \delta^{\min(2\gamma,1)} C e^{Ct}+C(t+1)\int_0^t Z(s)  ds.
\]
(For details about the derivation of this inequality, see Remark 5 of the first version of this paper \cite{JumKobayashi}.) 
The presence of $t+1$, which cannot be dominated by a constant since we do not impose a time horizon for the values of $t$, prevents the use of the classical result on Gronwall's inequality.  We could instead apply a generalized version of Gronwall's inequality in Theorem 1 on p.356 of \cite{Mitrinovic} (also see Remark 2 on p.357); however, this would yield the larger upper bound 
\[
	Z(t)
	\le \delta^{\min(2\gamma,1)} C e^{Ct}+C(t+1)\int_0^t \delta^{\min(2\gamma,1)} C e^{Cs} e^{\int_s^t C(r+1)  dr}  ds
	\le \delta^{\min(2\gamma,1)} C e^{C(t^2+t)}. 
\]
We emphasize here that the sharper bound in \eqref{Estimate_Euler2} is essential for the establishment of Theorems \ref{Theorem_main2} and \ref{Theorem_main3}; see item 3) of Remark \ref{Remark_main}.
\end{em}
\end{remark}

The next two results form the main components of the proof of Theorem \ref{Theorem_main4}, which concerns weak convergence of our approximation scheme.

\begin{lemma}\label{Lemma_weak1}
Let $X$ be the solution to SDE \eqref{SDE_classical} on the positive real line $[0,\infty)$ satisfying conditions \eqref{SDE_condition1} and \eqref{SDE_condition2}. 
Let $g\in C^2(\mathbb{R}^d)$ have derivatives of polynomial growth. 
Then for any $\delta\in(0,1)$ and any two time points $s$ and $t$ such that $0\le t-s\le \delta$, the inequality
\[
	|\mathbb{E}[g(X(t))-g(X(s))]|\le \delta C e^{Ct}
\]
holds, 
where $C$ is a constant not depending on $\delta$ or $t$. 
\end{lemma}

\begin{proof}
For simplicity, we give a proof only in the case $d=m=1$. 
By the It\^o formula, we obtain 
\[
	|\mathbb{E}[g(X(t))-g(X(s))]|
	\le \int_s^t \mathbb{E}[| g'(X(r))b(r,X(r))+\ g''(X(r))\sigma^2(r,X(r))/2|] dr.
\]
By condition \eqref{SDE_condition2} and the assumption that the derivatives of $g$ have polynomial growth, the quantity inside the expectation on the right hand side is dominated by a polynomial of $X(r)$. The desired upper bound now follows upon using the estimate \eqref{moment_estimate}.
\end{proof}

\begin{proposition}\label{Proposition_weak2}
Let $X$ be the solution to SDE \eqref{SDE_classical} on the positive real line $[0,\infty)$ with autonomous coefficients $b(x)$ and $\sigma(x)$ satisfying conditions \eqref{SDE_condition1} and \eqref{SDE_condition2}. Assume further that the coefficients are in $C^4(\mathbb{R}^d)$ and have derivatives of polynomial growth.
For a fixed $\delta\in(0,1)$, let $X_\delta$ be the process on $[0,\infty)$ defined in \eqref{def_Xdelta}--\eqref{def_interpolation}.  
Let $g\in C^4(\mathbb{R}^d)$ have derivatives of polynomial growth. 
Then there exists a positive constant $C$ not depending on $\delta$ or $t$ such that
\begin{align}\label{estimate_weak}
	\bigl|\mathbb{E}[g(X(t))-g(X_\delta(t))]\bigr|\le \delta C e^{Ct} \ \ \textrm{for all} \ \ t\ge 0. 
\end{align}
\end{proposition}

Note that application of Theorem 14.5.1 of \cite{KloedenPlaten} with $\beta=1$ provides the statement of this proposition but without specifying the upper bound in \eqref{estimate_weak} as a function of $t$. 
The first version of this paper \cite{JumKobayashi} gives a proof to clarify how the upper bound depending on $t$ emerges; however, since the idea used in the proof is similar to that of the proof of Theorem 14.1.5 of \cite{KloedenPlaten}, we omit the proof here and refer the interested readers to Proposition 7 of \cite{JumKobayashi}.

\subsection{Error estimates concerning approximations of inverse subordinators}

The following simple fact on inverse subordinators will play an important role in establishing the main results of this paper in 
Theorems \ref{Theorem_main2}, \ref{Theorem_main3} and \ref{Theorem_main4}.
It states that any inverse subordinator has finite exponential moment, which was originally proved in \cite{MagdziarzOW}. Here, we give an alternative proof which uses Laplace transform.

\begin{lemma}\label{Lemma_moments_general}
Let $E$ be the inverse of a subordinator $D$ with Laplace exponent $\psi$ in \eqref{def_LaplaceExponent} and infinite L\'evy measure. 
Then for all $\lambda\in\mathbb{R}$ and $t\ge 0$,
$\mathbb{E}[e^{\lambda E(t)}]<\infty$.
In particular, for each $t> 0$, moments of $E(t)$ of all orders exist and are given by 
\begin{align*}
	\mathbb{E}[E^n(t)]\bigr]
	=\mathcal{L}^{-1}_s\biggl[\dfrac{n!}{s\psi^n(s)}\biggr](t), \ \ n\in \mathbb{N},
\end{align*}
where $\mathcal{L}^{-1}_s[g(s)]$ denotes the inverse Laplace transform of a function $g(s)$.
\end{lemma}

\begin{proof}
Fix $x>0$. It follows from the inverse relationship between $D$ and $E$ that
\[
  \mathbb{P}(E(t)\le x)=\mathbb{P}(D(x)\ge t)
   =1-\mathbb{P}(D(x)< t)
\]
for $t>0$. 
Taking the Laplace transform with respect to $t$ on both sides and using \eqref{def_LaplaceExponent}, we obtain 
\[
	\mathcal{L}_t\bigl[\mathbb{P}(E(t)\le x)\bigr](s)
	=\dfrac{1}{s}-\dfrac{1}{s}\mathcal{L}_t[\mathbb{P}(D(x)\in dt)](s) 
	=\dfrac{1-e^{-x\psi(s)}}{s}
\]
for $s>0$,
where 
$\mathcal{L}_t[f(t)]$ and $\mathcal{L}_t[\mu(dt)]$ denote the Laplace transforms of a function $f(t)$ and a measure $\mu(dt)$, respectively. 
The right hand side of the above identity being differentiable with respect to $x$, so is the left hand side, and 
\begin{align*}
	\mathcal{L}_t[\mathbb{P}(E(t)\in dx)](s)
	=\dfrac{\psi(s)}{s}e^{-x\psi(s)} dx. 
\end{align*}
Hence, for a fixed $\lambda\in\mathbb{R}$ and for large $s>0$ such that $\psi(s)>\lambda$ (such $s$ necessarily exists since the L\'evy measure is assumed infinite), the Fubini theorem yields
\begin{align}\label{double_Laplace}
	\mathcal{L}_t \big[\mathbb{E}[e^{\lambda E(t)}]\bigr](s)
	= \int_0^\infty \dfrac{\psi(s)}{s}e^{-x(\psi(s)-\lambda)}  dx
	=\dfrac{\psi(s)}{s}(\psi(s)-\lambda)^{-1}<\infty.
\end{align}
This implies in particular that $\mathbb{E}[e^{\lambda E(t)}]<\infty$ for almost all $t>0$, but since the sample paths of $E$ are nondecreasing, this is indeed true for \textit{all} $t>0$. 
Therefore, for each fixed $t>0$, moments of $E(t)$ of all orders exist. 
Now, for a fixed $n\in\mathbb{N}$ and any $\lambda\in(-\epsilon,\epsilon)$, where $\epsilon>0$, we have the inequality 
$
	\mathbb{E}[E^n(t)e^{\lambda E(t)}]
	\le n! \mathbb{E}[e^{(1+\epsilon) E(t)}].
$
The right hand side is Laplace transformable as observed above, and therefore, taking derivatives with respect to $\lambda$ in the identity \eqref{double_Laplace}
yields
\begin{align*}
	\mathcal{L}_t\bigl[\mathbb{E}[E^n(t)e^{\lambda E(t)}]\bigr](s)
	=\dfrac{\psi(s)}{s}\dfrac{n!}{(\psi(s)-\lambda)^{n+1}}.
\end{align*}
Letting $\lambda \to 0$ and using the dominated convergence theorem (again due to the above estimate) gives
$
	\mathcal{L}_t\bigl[ \mathbb{E}[E^n(t)]\bigr](s)
	=n!/(s\psi^n(s)).
$
Taking the inverse Laplace transform completes the proof.
\end{proof}

\begin{example}\label{Remark_MittagLeffler}
\begin{em}
Let $D_\beta$ be a $\beta$-stable subordinator with $\beta\in(0,1)$ so that the Laplace exponent is given by $\psi(s)=s^\beta$. Let $E_\beta$ be the inverse of $D_\beta$. Then by Lemma \ref{Lemma_moments_general}, for each $t\ge 0$ and $n\in\mathbb{N}$, 
\begin{align}\label{moments_Salpha}
	\mathbb{E}[E_\beta^n(t)]
	=\mathcal{L}^{-1}_s\biggl[\dfrac{n!}{s^{n\beta+1}}\biggr](t)
	=\dfrac{n!}{\Gamma(n\beta+1)}t^{n\beta},
\end{align}
where $\Gamma(\cdot)$ is the Gamma function. 
Moreover, this implies that for all $\lambda\in\mathbb{R}$,  
$
	\mathbb{E}[e^{\lambda E_\beta(t)}]
	=\mathbf{E}_{\beta}(\lambda t^\beta),
$
where $\mathbf{E}_{\beta}(z):=\sum_{n=0}^\infty z^n/\Gamma(n\beta+1)$ is the Mittag--Leffler function. 
Therefore, Lemma \ref{Lemma_moments_general} can be regarded as a generalization of Proposition 1(a) iii) of \cite{Bingham}. 
\end{em}
\end{example}

Fix $\delta\in (0,1)$ and $T>0$.
To approximate an inverse subordinator $E$ on the interval $[0,T]$, 
we follow an idea presented in \cite{GajdaMagdziarz} to first simulate a sample path of the subordinator $D$, which has independent and stationary increments, by setting  $D(0)=0$ and then following the rule $D(i\delta):=D((i-1)\delta)+Z_i$, $i=1,2,3,\ldots,$
where $\{Z_i; i=1,2,\ldots\}$ is an i.i.d.\ sequence with $Z_i=^d D(\delta)$. We stop this procedure upon finding the integer $N$ satisfying 
\begin{align}\label{def_N}
	T\in[D(N\delta), D((N+1)\delta)).
\end{align}
 Note that the $\mathbb{N}\cup\{0 \}$-valued random variable $N$ indeed exists since  
 $D(t)\to\infty$ as $t\to\infty$ with probability one. 
To generate the random variables $\{Z_i\}$, one can use algorithms presented in Chapter 6 of \cite{ContTankov}; also consult \cite{BaeumerMeerschaert_temperedstable} for simulation of exponentially tempered stable random variables. 
Next, let
\begin{align}\label{def_Spsidelta}
	E_\delta(t)
	:=\bigl(\min\{n\in \mathbb{N}; D(n\delta)>t\}-1\bigr)\delta, \ \ t\in[0,T].
\end{align}
The sample paths of $E_\delta$ are nondecreasing step functions with constant jump size $\delta$ and the $i$th waiting time given by $Z_i=D(i\delta)-D((i-1)\delta)$. Indeed, it is easy to see that for $n=0,1,2,\ldots,N$,
\begin{align}\label{property_Spsidelta}
	E_\delta(t)=n\delta \ \ \textrm{whenever} \ \ t\in[D(n\delta),D((n+1)\delta)).
\end{align}
In particular, \eqref{def_N} is equivalent to 
\begin{align}\label{property_N}
	E_\delta(T)=N\delta.
\end{align}
The process $E_\delta$ efficiently approximates $E$, as the following lemma shows.

\begin{lemma}\label{Lemma_Spsi_approx}
Let $E$ be the inverse of a subordinator $D$ with infinite L\'evy measure. Let $E_\delta$ be the process defined in \eqref{def_Spsidelta}.
Then with probability one, 
\begin{align}\label{ineq_Spsi}
	E(t)-\delta\le E_\delta(t)\le E(t) \ \ \textrm{for all} \ \  t\in[0,T]. 
\end{align}
\end{lemma}

\begin{proof}
An original proof of this lemma is due to \cite{Magdziarz_spa}. Here, we give a slightly different but simple argument to obtain the same result. 

Since the sample paths of $E$ are continuous and nondecreasing and satisfy 
\begin{align}\label{property_Spsi}
	E(D(n\delta))=n\delta,  n=0,1,2,\ldots, 
\end{align}	
comparison of \eqref{property_Spsidelta} and \eqref{property_Spsi} immediately gives the desired result. 
\end{proof}

\subsection{Main results --- strong and weak convergence along with their respective orders}

By the duality principle (Lemma \ref{Lemma_duality}), the solution $Y$ to SDE \eqref{SDE_new} on a fixed interval $[0,T]$ can be expressed as $X\circ E$ with $X$ denoting the solution to SDE \eqref{SDE_classical} on $[0,\infty)$ and $E$ denoting an inverse subordinator. Hence, it is reasonable to approximate $Y$ by the process $Y_\delta$ defined by 
\begin{align}\label{def_Ypsi}
	Y_\delta(t):=X_\delta(E_\delta(t)), \ \ t\in[0,T],
\end{align}
where $X_\delta$ and $E_\delta$ are the processes defined in \eqref{def_Xdelta}--\eqref{def_interpolation} and \eqref{def_Spsidelta}, respectively. Note that we consider $X$ and $X_\delta$ on the positive real line $[0,\infty)$ (rather than on a finite interval), and hence, the expressions $Y=X\circ E$ and $Y_\delta=X_\delta\circ E_\delta$ are both meaningful even though $E$ and $E_\delta$ can take all values in $[0,\infty)$. 
This is why 
we established Propositions \ref{Proposition_A(T)} and \ref{Proposition_weak2} with the time interval being $[0,\infty)$.

On the other hand, $Y$ and its approximation $Y_\delta$ in \eqref{def_Ypsi} are defined on a finite interval $[0,T]$. At the time horizon, the process $Y_\delta$ takes the value $Y_\delta(T)=X_\delta(N\delta)$ due to \eqref{property_N}. Hence, to generate a sample path of $Y_\delta$, we first find the integer $N$ satisfying \eqref{def_N} and then construct $X_\delta$ on the bounded interval $[0,N\delta]$ using the finitely many discretization points $\{0,\delta,2\delta,\ldots,N\delta\}$. 
Details on how to conduct simulation will be summarized in Section \ref{section_simulation}.

Now, a natural question to ask is whether $Y_\delta$ converges to $Y$ in some reasonable sense as $\delta\to 0$ and, if so, what the rate of convergence is. The following theorems answer this question.

\begin{theorem}\label{Theorem_main2}
Let $B$ be an $m$-dimensional Brownian motion independent of a subordinator $D$ with infinite L\'evy measure with inverse $E$.
Let $Y$ be the solution to SDE \eqref{SDE_new} on a fixed interval $[0,T]$ satisfying conditions \eqref{SDE_condition1}, \eqref{SDE_condition2} and \eqref{SDE_condition3}. For a fixed $\delta\in(0,1)$, let $Y_\delta$ be the process defined in \eqref{def_Ypsi}. Then 
\begin{align*}
	\mathbb{E}[|Y(T)-Y_\delta(T)|^2]
	\le C\delta^{\min(2\gamma,1)},
\end{align*}
where $C$ is a positive constant not depending on $\delta$. In particular, $Y_\delta$ converges strongly to $Y$ at the time horizon $T$ with order $\min(\gamma,1/2)$; that is, 
\begin{align}\label{strong_convergence}
	\mathbb{E}[|Y(T)-Y_\delta(T)|]
	\le C^{1/2}\delta^{\min(\gamma,1/2)}.
\end{align}
\end{theorem}

\begin{proof}
Note that the assumption that $B$ is independent of $D$ implies that the vectors $(X,X_\delta)$ and $(E, E_\delta)$ are independent.
By \eqref{ineq_Spsi}, Proposition \ref{Proposition_A(T)}, and the independence assumption, we observe that
\begin{align}\label{estimate_main1}
	\mathbb{E}\biggl[\sup_{0\le t\le T}
		|X(E_\delta(t))-X_\delta(E_\delta(t))|^2\biggr]
	&\le \mathbb{E}\biggl[\sup_{0\le s\le E(T)}
		|X(s)-X_\delta(s)|^2\biggr]\\
	&\le C_1\mathbb{E}[e^{C_1E(T)}]\delta^{\min(2\gamma,1)},\notag
\end{align}
where $C_1$ is a positive constant not depending on $\delta$. 
On the other hand, Lemma \ref{Lemma_modulus} along with \eqref{ineq_Spsi} and independence implies that, for some $C_2>0$,
\begin{align}\label{estimate_main2}
	\mathbb{E}[|X(E(T))-X(E_\delta(T))|^2]
	&\le C_2\mathbb{E}[e^{C_2 E(T)}]\delta. 
\end{align}
Hence, by the triangle inequality, $\mathbb{E}[|Y(T)-Y_\delta(T)|^2]$ is dominated by 
\begin{align*}
	&2\mathbb{E}[|X(E(T))-X(E_\delta(T))|^2]
	+ 2\mathbb{E}[|X(E_\delta(T))-X_\delta(E_\delta(T))|^2]\\
	&\le 2C_2\mathbb{E}[e^{C_2 E(T)}]\delta
	+ 2C_1\mathbb{E}[e^{C_1E(T)}]\delta^{\min(2\gamma,1)}.
\end{align*}
 The desired estimate now follows due to Lemma \ref{Lemma_moments_general}.
 \end{proof}

\begin{remark}\label{Remark_main}
\begin{em}
1) Our approximation scheme extends the scheme presented in Section III of \cite{GajdaMagdziarz} to SDEs of the form \eqref{SDE_new} with general time-dependent coefficients and inverse subordinators. Moreover, that paper does not discuss the order of convergence of $Y_\delta$ to $Y$. 
Thus, the result established in Theorem \ref{Theorem_main2} of this paper is completely new.

2) The argument given in item 1) of Remark \ref{Remark_Euler} also applies to Theorem \ref{Theorem_main2} (and Theorem \ref{Theorem_main3} below as well). Namely, to guarantee the statement of Theorem \ref{Theorem_main2} to hold for all $\delta\in(0,\delta_0)$ for a fixed $\delta_0\in (0,1]$, it is sufficient to assume condition \eqref{SDE_condition3} only for $s,t\ge 0$ satisfying $|s-t|\le \delta_0$.

3) Recall item 2) of Remark \ref{Remark_Euler}, where we emphasized that our proof of Proposition \ref{Proposition_A(T)} gives a sharper bound in \eqref{Estimate_Euler2} (i.e.\ $Ce^{Ct}\delta^{\min(2\gamma,1)}$) than the bound that would be obtained by a simple modification of the well-known proof of Lemma \ref{Lemma_Euler} in \cite{KloedenPlaten} (i.e.\ $Ce^{C(t^2+t)}\delta^{\min(2\gamma,1)}$). Note that the rougher bound would \textit{not} be sufficient to establish Theorem \ref{Theorem_main2} for \textit{general} inverse subordinators 
since the expectation $\mathbb{E}[e^{C (E^2(T)+E(T))}]$ may be infinite, and hence, the upper bound in \eqref{estimate_main1} may be meaningless.
For example, consider the inverse $E_\beta$ of a $\beta$-stable subordinator $D_\beta$ with $\beta\in(0,1)$ discussed in Example \ref{Remark_MittagLeffler}. 
By the formula in \eqref{moments_Salpha}, 
\begin{align*}
	\mathbb{E}[e^{\lambda E_\beta^2(T)}]
	=\sum_{n=0}^\infty \dfrac{\lambda^n\mathbb{E}[E_\beta^{2n}(T)]}{n!}
	=\sum_{n=0}^\infty\dfrac{\lambda^n}{n!}\dfrac{(2n)!}{\Gamma(2n\beta+1)}T^{2n\beta}
	=f(\lambda T^{2\beta}),
\end{align*}
where $f(z):=\sum_{n=0}^\infty a_n z^n$, $z\in\mathbb{C}$, with 
$a_n:=(2n)!/(n!\Gamma(2n\beta+1))$.
By Stirling's formula 
$\Gamma(x+1)\sim \sqrt{2\pi} x^{x+1/2}e^{-x}$ as $x\to\infty$,
it follows that, as $n\to\infty$,
\begin{align*}
	\dfrac{a_{n+1}}{a_n}
	&=\dfrac{(2n+2)!}{(n+1)!\Gamma(2(n+1)\beta+1)}\cdot 
		\dfrac{n!\Gamma(2n\beta+1)}{(2n)!}\\
	&\sim 2(2n+1)\cdot 
		\dfrac{(2n\beta)^{2n\beta+1/2}e^{-2n\beta}}
		{(2(n+1)\beta)^{2(n+1)\beta+1/2}e^{-2(n+1)\beta}}
	\sim \dfrac{2(2n+1)}{(2(n+1)\beta)^{2\beta}}.
\end{align*}
If $\beta\in(0,1/2)$, then the last expression diverges to infinity, and hence, the power series $f(z)$ converges only at $z=0$ due to the ratio test. 
Consequently, $\mathbb{E}[e^{\lambda E_\beta^2(T)}]=f(\lambda T^{2\beta})=\infty$ for all $\lambda>0$, which implies that $\mathbb{E}[e^{C (E_\beta^2(T)+E_\beta(T))}]=\infty$.

4) Instead of the Euler--Maruyama scheme, it is possible to use higher order strong It\^o--Taylor approximation schemes to construct a process approximating the solution $X$ of SDE \eqref{SDE_classical} (see Section 10.6 of \cite{KloedenPlaten}), but that does not improve the order of strong convergence of $Y_\delta$ to $Y$ 
 since the estimate \eqref{estimate_main2} remains unchanged.
 
 5) Paper \cite{Magdziarz_spa} suggests the use of a time-dependent drift coefficient of the form $b(D(r),X(r))$ in place of $b(r,X(r))$ in SDE  \eqref{SDE_classical}, where $D$ is a general subordinator. In this case, the method presented in this paper cannot be applied to obtain a convergence result regarding approximation of the process $Y=X\circ E$. In fact, $X$ defined via the drift coefficient $b(D(r),X(r))$ clearly depends on $D$ (and hence on $E$ as well); consequently, the conditioning argument in \eqref{estimate_main1} is no longer valid. 
\end{em}
\end{remark}

The next theorem shows that the strong convergence of $Y_\delta$ to $Y$ discussed in Theorem \ref{Theorem_main2} actually occurs \textit{uniformly} over the entire interval $[0,T]$. However, the proof provided below, which utilizes a result on modulus of continuity for stochastic integrals in \cite{FischerNappo}, does not provide the exact order of convergence.

\begin{theorem}\label{Theorem_main3}
Let $B$ be an $m$-dimensional Brownian motion independent of a subordinator $D$ with infinite L\'evy measure with inverse $E$.
Let $Y$ be the solution to SDE \eqref{SDE_new}  on a fixed interval $[0,T]$ satisfying conditions \eqref{SDE_condition1}, \eqref{SDE_condition2} and \eqref{SDE_condition3}. For $\delta\in(0,1)$, let $Y_\delta$ be the process defined in \eqref{def_Ypsi}. Then 
$Y_\delta$ converges strongly to $Y$ uniformly on $[0,T]$ in $L^2$; i.e.\ 
 \[
 	\lim_{\delta\to 0}\mathbb{E}[\sup_{0\le t\le T}|Y(t)-Y_\delta(t)|^2]=0.
 \]
\end{theorem}

\begin{proof}
For a fixed $u>0$, since $x\le \sqrt{2u x \log(2u/x)}$ for $0<x<u$, the $i$th component $b_i(t,x)$ of the drift coefficient of SDE \eqref{SDE_classical} satisfies 
\begin{align*}
	\int_s^t |b_i(r,X(r))|  dr
	\le (t-s)K\Bigl(1+\sup_{r\in[0,u]}|X(r)|\Bigr)
	\le \zeta\sqrt{(t-s)\log \Bigl(\dfrac{2u}{t-s}\Bigr)}
\end{align*}
for all $0\le s< t\le u$, where $\zeta:=\sqrt{2u}K(1+\sup_{r\in[0,u]}|X(r)|)$. On the other hand, the $(i,j)$th component $\sigma_{i,j}(t,x)$ of the diffusion coefficient satisfies the inequality
\[
	\int_s^t \sigma_{i,j}^2(r,X(r))  dr\le \xi(t-s)
\]
 for all $0\le s< t\le u$, where $\xi:=2K^2(1+\sup_{r\in[0,u]}|X(r)|^2)$. Exercise 4.5.5 of \cite{KloedenPlaten} shows that both $\zeta$ and $\xi$ have moments of all orders. Hence, application of Theorem 1 of \cite{FischerNappo}, which concerns the modulus of continuity of stochastic integrals driven by Brownian motion with drift, implies that there exists a constant $C$ such that
 \begin{align}\label{proof_uniform}
 	\mathbb{E}\biggl[\Bigl(\sup_{r, s\in[0, u],\ 0\le s-r\le \delta}|X(s)-X(r)|\Bigr)^2\biggr]
	\le C\delta\log\Bigl(\dfrac{2u}{\delta}\Bigr)
 \end{align}
for all $\delta\in(0, u]$.
Note that the proof of Theorem 1 of \cite{FischerNappo} shows that the constant $C$ in \eqref{proof_uniform} can be taken independently of $\delta$ and the fixed time horizon $u$. 
Using \eqref{ineq_Spsi}, we observe that $\mathbb{E}[(\sup_{0\le t\le T}|X(E(t))-X(E_\delta(t))|)^2]$ is dominated by 
\begin{align*}
	&\mathbb{E}\biggl[\Bigl(\sup_{r, s\in[0, E(T)],\ 0\le s-r\le \delta}|X(s)-X(r)|\Bigr)^2\biggr]\\
	&\le \mathbb{E}\biggl[\Bigl(\sup_{r, s\in[0, E(T)],\ 0\le s-r\le \delta}|X(s)-X(r)|\Bigr)^2\mathbf{1}_{\{E(T)\ge \delta\}}\biggr]\\
	&\ \ \ +\mathbb{E}\biggl[\Bigl(\sup_{r, s\in[0, \delta],\ 0\le s-r\le \delta}|X(s)-X(r)|\Bigr)^2\biggr],
\end{align*}
where $\mathbf{1}_U$ denotes the indicator function of a set $U$. Hence, it follows from the independence assumption and the estimate \eqref{proof_uniform} with $u=E(T)$ and $u=\delta$ that
\begin{align}\label{proof_uniform2}
	\mathbb{E}\biggl[\Bigl(\sup_{0\le t\le T}|X(E(t))-X(E_\delta(t))|\Bigr)^2\biggr]
	&\le C \mathbb{E}\biggl[\delta \log \Bigl(\dfrac{2E(T)}{\delta}\Bigr)\biggr]
		+C\delta \log 2\\
	&=C \mathbb{E}\biggl[\delta \log \Bigl(\dfrac{4E(T)}{\delta}\Bigr)\biggr].\notag
\end{align}
By the triangle inequality and the estimates \eqref{estimate_main1} and \eqref{proof_uniform2},
\begin{align*}
	&\mathbb{E}\biggl[\sup_{0\le t\le T}|Y(t)-Y_\delta(t)|^2\biggr]\\
	&\le 
	2\mathbb{E}\biggl[\sup_{0\le t\le T}|X(E(t))-X(E_\delta(t))|^2\biggr]
	+2\mathbb{E}\biggl[\sup_{0\le t\le T}|X(E_\delta(t))-X_\delta(E_\delta(t))|^2\biggr]\\
	&\le 
	2C \mathbb{E}\biggl[\delta \log \Bigl(\dfrac{4E(T)}{\delta}\Bigr)\biggr]
	+ 2C_1\mathbb{E}[e^{C_1 E(T)}]\delta^{\min(2\gamma,1)}.
\end{align*}
Now, the obvious inequality $\log x<x$ for $x>0$ together with Lemma \ref{Lemma_moments_general} allows the use of the dominated convergence to yield 
$
	\lim_{\delta\to 0}\mathbb{E}[\delta \log (4E(T)/\delta)]= 0,
$
which completes the proof. 
\end{proof}

Many practical situations do not require so strong a convergence of $Y_\delta$ to $Y$ as in Theorems \ref{Theorem_main2} and \ref{Theorem_main3}, but may only need e.g.\ computation of moments at the time horizon $T$. In such cases, it is more reasonable to look for an upper bound for the quantity $\bigl|\mathbb{E}[g(Y(T))-g(Y_\delta(T))]\bigr|$ for some function $g$ rather than the pathwise error estimate in \eqref{strong_convergence}.  
We know a priori from Theorem \ref{Theorem_main2} and the mean value theorem that, as long as $g$ is a function with a bounded derivative, the estimate
$
	\bigl|\mathbb{E}[g(Y(T))-g(Y_\delta(T))]\bigr|
	\le C\delta^{\min(\gamma,1/2)}
$
holds.
 However, as the following theorem shows, 
 the upper bound can be improved under some smoothness assumptions on the function $g$ and the coefficients of SDE \eqref{SDE_new}.

\begin{theorem}\label{Theorem_main4}
Let $B$ be an $m$-dimensional Brownian motion independent of a subordinator $D$ with infinite L\'evy measure with inverse $E$.
Let $Y$ be the solution to SDE \eqref{SDE_new} on a fixed interval $[0,T]$ 
 with autonomous coefficients $b(x)$ and $\sigma(x)$ satisfying conditions \eqref{SDE_condition1} and \eqref{SDE_condition2}. Assume further that the coefficients are in $C^4(\mathbb{R}^d)$ and have derivatives of polynomial growth.
For a fixed $\delta\in(0,1)$, let $Y_\delta$ be the process defined in \eqref{def_Ypsi}. 
Let $g\in C^4(\mathbb{R}^d)$ have derivatives of polynomial growth. 
Then 
\begin{align}\label{weak_convergence}
	\bigl|\mathbb{E}[g(Y(T))-g(Y_\delta(T))]\bigr|
	\le C\delta,
\end{align}
where $C$ is a positive constant not depending on $\delta$;
thus, $Y_\delta$ converges weakly to $Y$ at the time horizon $T$ with order $1$.
\end{theorem}

\begin{proof}
By the triangle inequality,  $|\mathbb{E}[g(Y(T))-g(Y_\delta(T))]|$ is dominated by 
$
	|\mathbb{E}[g(X(E(T)))-g(X(E_\delta(T)))]|
		+|\mathbb{E}[g(X(E_\delta(T)))-g(X_\delta(E_\delta(T)))]|.
$
Using Lemma \ref{Lemma_weak1} and Proposition \ref{Proposition_weak2} together with the independence assumption, we obtain 
\begin{align}\label{proof_weak0}
	|\mathbb{E}[g(Y(T))-g(Y_\delta(T))]|
	\le \delta C_1 \mathbb{E}[e^{C_1 E(T)}]+\delta C_2 \mathbb{E}[e^{C_2 E_\delta(T)}].
\end{align}
Since the expectations on the right hand side are finite due to \eqref{ineq_Spsi} and Lemma \ref{Lemma_moments_general}, the proof is complete.
\end{proof}

\begin{remark}\label{Remark_main2}
\begin{em}
1) Proposition \ref{Proposition_weak2} and Theorem \ref{Theorem_main4} apply to non-autonomous cases as well, which require additional smoothness assumptions on the coefficients (for details of this matter, see a discussion following Theorem 14.5.1 of \cite{KloedenPlaten}). 

2) 
The smoothness assumption on $g$ may create issues in some applications. For example, to price a European call option with the underlying stock price following a time-changed analogue of a Black--Scholes SDE, $g$ should be taken to be $g(x):=\max(x-K_0,0)$ for some constant $K_0$ (for details of option pricing and Black--Scholes SDEs, see e.g.\ \cite{Steele}). One way to deal with such situations is to apply Theorem \ref{Theorem_main4} to some smooth functions approximating the non-smooth function $g$. 

3) Using higher order weak It\^o--Taylor schemes instead of the Euler--Maruyama scheme (see Section 14.5 of \cite{KloedenPlaten} for details) does not improve the order of weak convergence of $Y_\delta$ to $Y$ 
 since the first term on the right hand side of the estimate \eqref{proof_weak0} remains unchanged. 
\end{em}
\end{remark}

\section{Numerical examples}\label{section_simulation}

For a given $\delta\in(0,1)$, a sample path of the process $Y_\delta$ on a fixed interval $[0,T]$, which approximates the solution $Y$ to SDE \eqref{SDE_new}, is generated by the following simple steps: 
\begin{enumerate}
	\item Simulate $D$ at the discretization points $\{0,\delta,2\delta,\ldots\}$ and stop this procedure upon finding an integer $N$ satisfying $T\in [D(N\delta),D((N+1)\delta)$. 
	\item Simulate $X_\delta$ using the Euler--Maruyama scheme at the finitely many discretization points $\{0,\delta,2\delta,\ldots,N\delta\}$. 
	\item Based on 1) and 2), set 
		\begin{itemize}
			\item $Y_\delta(t)=X_\delta(n\delta)$ for $t\in [D(n\delta), D((n+1)\delta)$ with $n=0,1,2,\ldots,N-1$;
			\item $Y_\delta(t)=X_\delta(N\delta)$ for $t\in [D(N\delta), T]$.
		\end{itemize}
\end{enumerate}
Note that the continuously interpolated values of $X_\delta$ defined in \eqref{def_interpolation} are never used in the above simulation steps; the interpolation was introduced solely for the purpose of deriving Propositions \ref{Proposition_A(T)} and \ref{Proposition_weak2}.

As a simple example with which to numerically verify the statements of Theorems \ref{Theorem_main2} and \ref{Theorem_main4}, consider the SDE 
\[
	Y(t)=1+\int_0^t Y(s)  dB(E(s)), \ \ t\in[0,1],
\]
with $B$ being a one-dimensional Brownian motion and $E$ being the inverse of an independent exponentially tempered stable subordinator $D$ whose L\'evy measure is given by $\nu(dx)=(e^{-\kappa x}/x^{1+\beta}) \mathbf{1}_{x>0} dx$,
where $\beta\in(0,1)$ is the stability index and $\kappa>0$ is a tempering factor. 
Here we fix $\beta=0.95$ and $\kappa=1$ and  
employ an algorithm presented in \cite{BaeumerMeerschaert_temperedstable} to generate sample paths of $D$. 
The solutions of SDE \eqref{SDE_classical} and SDE \eqref{SDE_new} in this case are respectively given by $X(t)=e^{B(t)-t/2}$ and $Y(t)=X(E(t))=e^{B(E(t))-E(t)/2}$. 
Note that it is not possible to generate sample paths of the \textit{exact} solution $Y=X\circ E$ since there is no way to realize sample paths of the \textit{exact} time change $E$. 
With this in mind, we compare in Figure \ref{fig:one} the sample path behavior of the ``near-exact'' solution $X\circ E_\delta$ (instead of the exact solution $Y=X\circ E$) with that of the approximation process $Y_\delta=X_\delta\circ E_\delta$ with the equidistant step size $\delta=10^{-3}$, where the underlying path of the discretized time change $E_\delta$ is also provided for reference. 
Note that because of the way the processes are constructed, the three trajectories share the same constant periods.

\begin{figure}[htbp]
 \begin{center}
  \includegraphics[bb= 0 0 587 388,width=115mm]{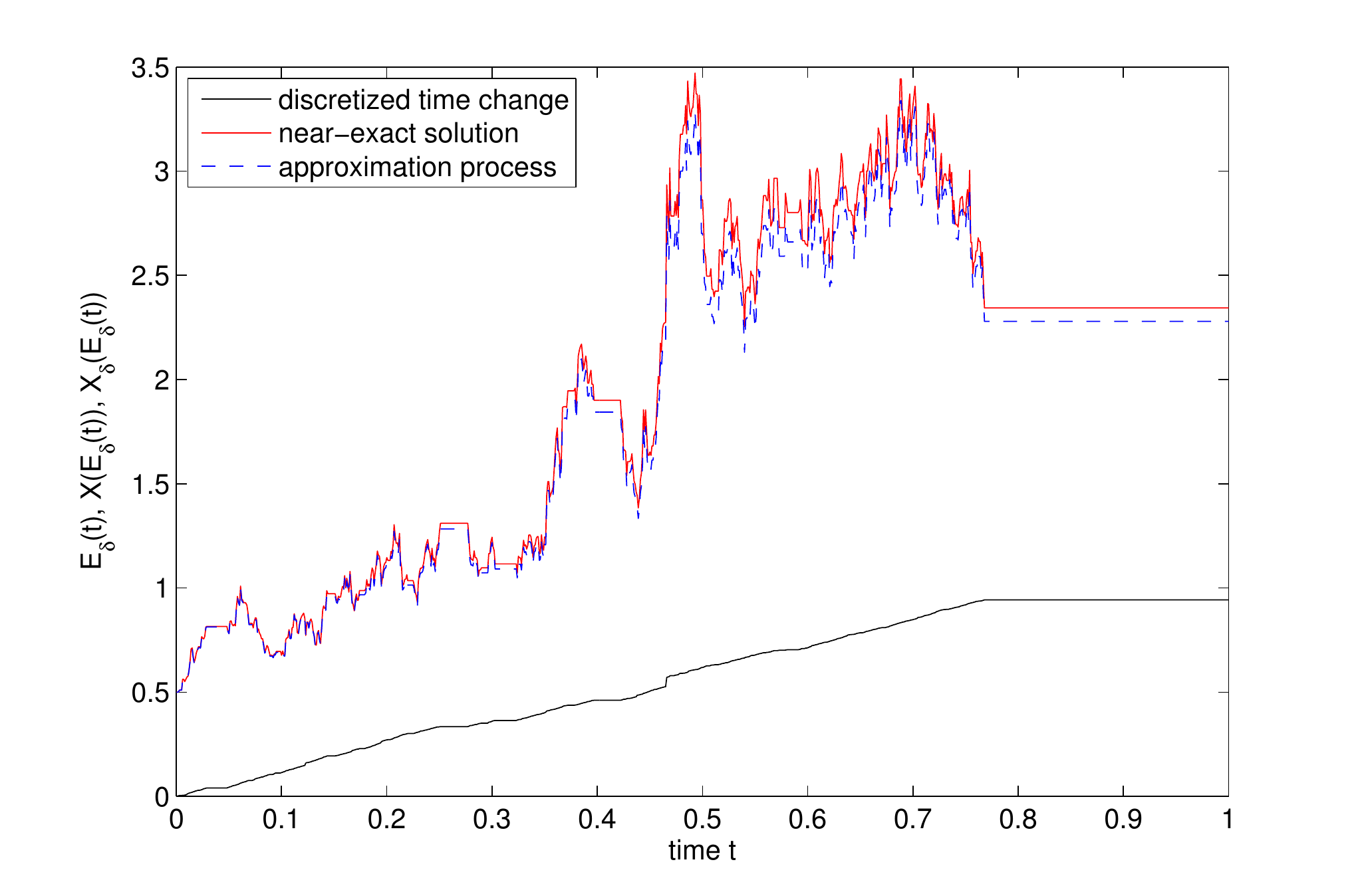}
 \caption{Comparison of sample paths of the near-exact solution $X\circ E_\delta$ and the approximation process $X_\delta\circ E_\delta$ along with the underlying sample path of the discretized time change $E_\delta$.} 
 \label{fig:one}
 \end{center}
\end{figure}

To carefully examine the order of convergence, 
for different values of $\delta$, we generated 300 sample paths for each of the near-exact solution and the approximation. We then calculated the following two errors at the time horizon $T=1$:
\begin{align*}
	\textrm{STERR}(\delta)&:=\dfrac 1{300}\sum_{i=1}^{300}|X(E_\delta(T))(\omega_i)-X_\delta(E_\delta(T))(\omega_i)|;\\
	\textrm{WKERR}(\delta)&:=\biggl|\dfrac 1{300}\sum_{i=1}^{300}X(E_\delta(T))(\omega_i)-\dfrac 1{300}\sum_{i=1}^{300}X_\delta(E_\delta(T))(\omega_i)\biggr|,
\end{align*}
where $\omega_i$ represents the $i$th realization. Here, STERR$(\delta)$ and WKERR$(\delta)$ are unbiased estimates for the theoretical errors involved with strong convergence in \eqref{strong_convergence} and weak convergence with $g(x)=x$ in \eqref{weak_convergence}, respectively. Namely, STERR$(\delta)$ gives an estimate for 
$\mathbb{E}[|X(E_\delta(T))-X_\delta(E_\delta(T))|],$
which is dominated by $C\delta^{1/2}$ due to Theorem \ref{Theorem_main2}, while WKERR$(\delta)$ is for $\bigl|\mathbb{E}[X(E_\delta(T))-X_\delta(E_\delta(T))]\bigr|,$
which has the upper bound $C\delta$ by Theorem \ref{Theorem_main4}.

\begin{figure}[tbp]
 \begin{minipage}{0.48\hsize}
  \hspace{-13mm} \includegraphics[bb= 0 0 587 276, width=100mm]{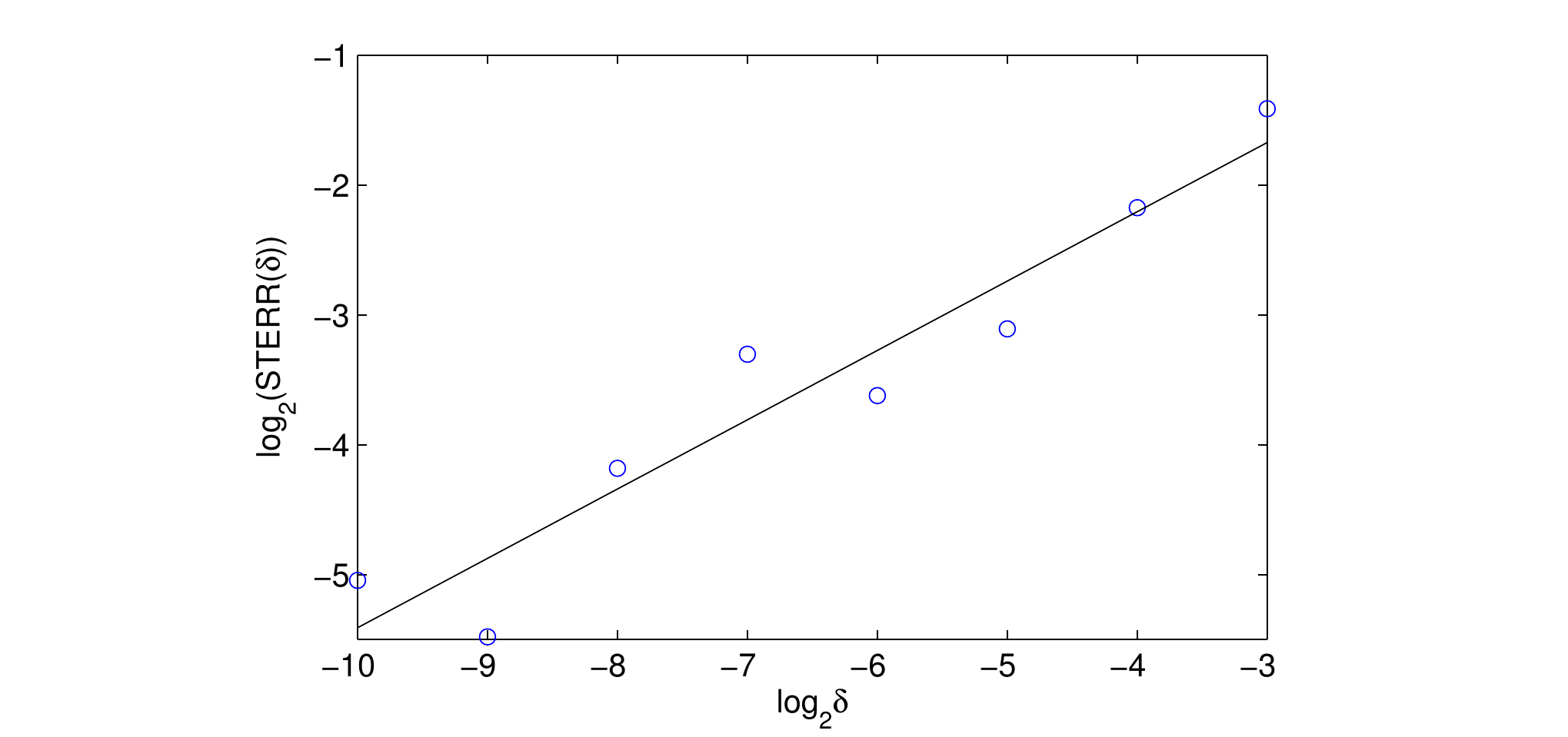}
  \caption{Plot of $\log_2 (\textrm{STERR}(\delta))$ versus $\log_2 \delta$ with the least squares line $y=0.5338\hspace{1pt}x-0.0719$.}
  \label{fig:two}
 \end{minipage}
 \hspace{5mm}
 \begin{minipage}{0.48\hsize}
 \hspace{-13mm} \includegraphics[bb= 0 0 587 276,width=100mm]{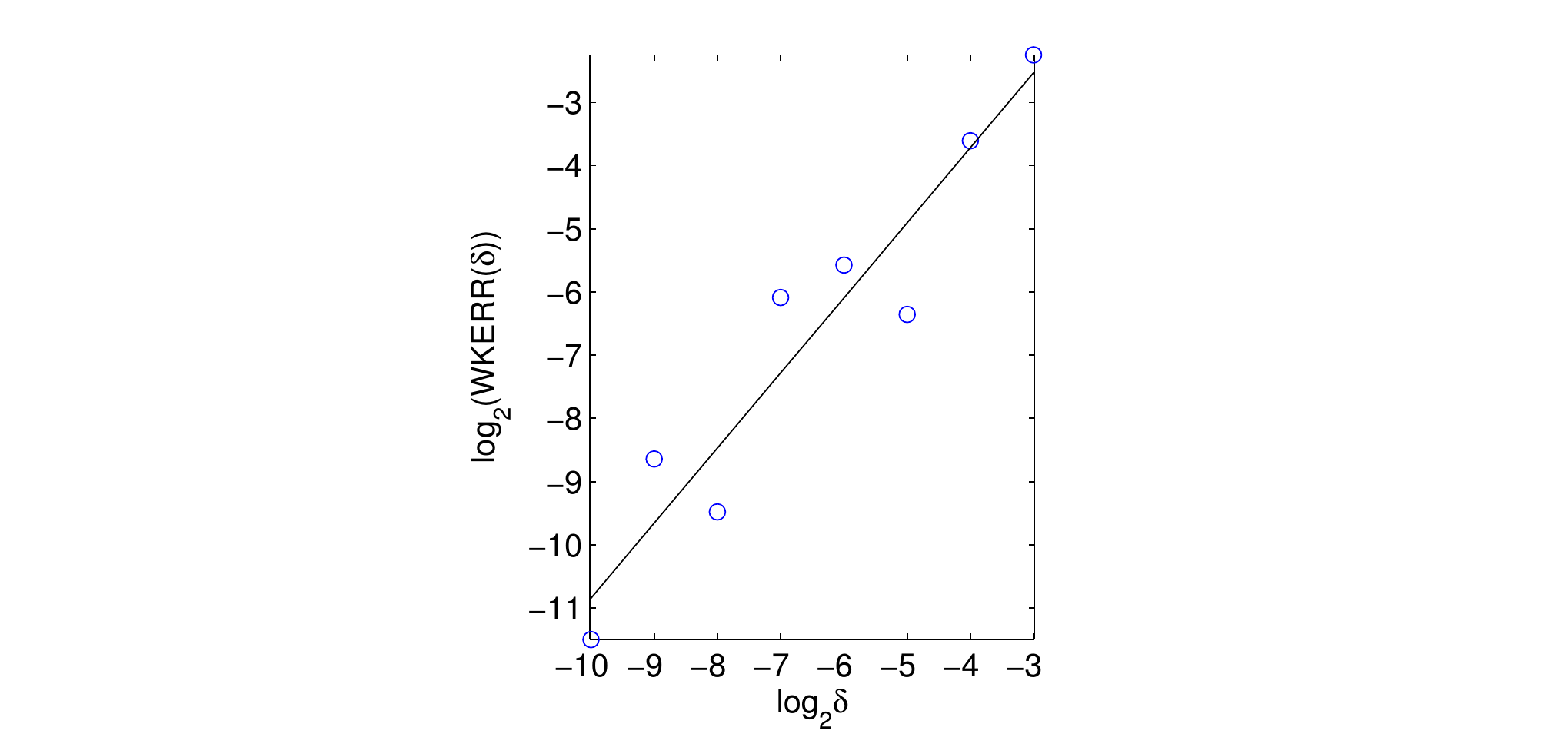}
  \caption{Plot of $\log_2 (\textrm{WKERR}(\delta))$ versus $\log_2 \delta$ with the least squares line $y=1.1882\hspace{1pt}x+1.0362$.}
  \label{fig:three}
 \end{minipage}
\end{figure}

Figure \ref{fig:two} gives a plot of $\log_2 (\textrm{STERR}(\delta))$ against $\log_2 \delta$.
It shows a linear trend with least squares slope being 0.5338. This is slightly higher than 0.5, which is the largest possible slope suggested by the estimate \eqref{strong_convergence} . 
On the other hand, Figure \ref{fig:three} provides a plot of $\log_2 (\textrm{WKERR}(\delta))$ versus $\log_2 \delta$, for which the least squares slope turns out to be  1.1882. This is close to 1.0 as suggested by \eqref{weak_convergence}. 
As the number of paths generated increases, the corresponding least squares slopes are expected to approach 0.5 and 1.0, respectively.

\vspace{8mm}

\noindent
{\textbf{Acknowledgments:}}
The authors would like to thank Professor Jan Rosi\'nski at the University of Tennessee for helpful discussions. 
We also appreciate the comments and suggestions of an anonymous referee which helped to improve the paper. This paper was written while E.\ Jum was a graduate student at the Department of Mathematics of the University of Tennessee. He thanks the people in the department for their warm support and encouragement.

  \bibliographystyle{plain} 
  \bibliography{KobayashiK_20151111}

\begin{thebibliography}{10}

\bibitem{BaeumerMeerschaert_temperedstable}
B.~Baeumer and M.~M. Meerschaert.
\newblock Tempered stable {L}{\'e}vy motion and transient super-diffusion.
\newblock {\em \textit{J. Comput. Appl. Math.}}, 233:2438--2448, 2010.

\bibitem{BWM}
D.~A. Benson, S.~W. Wheatcraft, and M.~M. Meerschaert.
\newblock Application of a fractional advection-dispersion equation.
\newblock {\em \textit{Water Resour. Res.}}, 36(6):1403--1412, 2000.

\bibitem{Bingham}
N.~H. Bingham.
\newblock Limit theorems for occupation times of markov processes.
\newblock {\em \textit{Z. Wahrsch. Verw. Gebiete}}, 17:1--22, 1971.

\bibitem{Burr}
M.~Burr.
\newblock Weak convergence of stochastic integrals driven by continuous-time
  random walks.
\newblock 2011.
\newblock \texttt{ar{X}iv:1110.0216 [math.{PR}]}.

\bibitem{ContTankov}
R.~Cont and P.~Tankov.
\newblock {\em \textit{Financial Modelling with Jump Processes}}.
\newblock Chapman and Hall/CRC, 2003.

\bibitem{FischerNappo}
M.~Fischer and G.~Nappo.
\newblock On the moments of the modulus of continuity of it{\^o} processes.
\newblock {\em \textit{Stoch. Anal. Appl.}}, 28(1):103--122, 2008.

\bibitem{GajdaMagdziarz}
J.~Gajda and M.~Magdziarz.
\newblock Fractional {F}okker--{P}lanck equation with tempered $\alpha$-stable
  waiting times: {L}angevin picture and computer simulation.
\newblock {\em \textit{Phys. Rev. E}}, 82:011117, 2010.

\bibitem{GMSR}
R.~Gorenflo, F.~Mainardi, E.~Scalas, and M.~Raberto.
\newblock Fractional calculus and continuous-time finance {III}: the diffusion
  limit.
\newblock {\em \textit{Mathematical Finance}, Trends in Mathematics}, pages
  171--180, 2001.

\bibitem{HKU-1}
M.~G. Hahn, K.~Kobayashi, and S.~Umarov.
\newblock {SDE}s driven by a time-changed {L}{\'e}vy process and their
  associated time-fractional order pseudo-differential equations.
\newblock {\em \textit{J. Theoret. Probab.}}, 25(1):262--279, 2012.

\bibitem{Jacod}
J.~Jacod.
\newblock {\em \textit{Calcul Stochastique et Probl{\`e}mes de Martingales}},
  volume 714 of {\em Lecture Notes in Mathematics}.
\newblock Springer, Berlin, 1979.

\bibitem{JumKobayashi}
E.~Jum and K.~Kobayashi.
\newblock A strong and weak approximation scheme for stochastic differential
  equations driven by a time-changed brownian motion.
\newblock 2014.
\newblock \texttt{ar{X}iv:1408.4377v1 [math.{PR}]}.

\bibitem{KloedenPlaten}
P.~E. Kloeden and E.~Platen.
\newblock {\em \textit{Numerical Solution of Stochastic Differential
  Equations}}.
\newblock Springer, corrected edition, 1992.

\bibitem{Kobayashi}
K.~Kobayashi.
\newblock Stochastic calculus for a time-changed semimartingale and the
  associated stochastic differential equations.
\newblock {\em \textit{J. Theoret. Probab.}}, 24(3):789--820, 2011.

\bibitem{Magdziarz_BS}
M.~Magdziarz.
\newblock Black--{S}choles formula in subdiffusive regime.
\newblock {\em \textit{J. Stat. Phys.}}, 136:553--564, 2009.

\bibitem{Magdziarz_simulation}
M.~Magdziarz.
\newblock Langevin picture of subdiffusion with infinitely divisible waiting
  times.
\newblock {\em \textit{J. Stat. Phys.}}, 135:763--772, 2009.

\bibitem{Magdziarz_spa}
M.~Magdziarz.
\newblock Stochastic representation of subdiffusion processes with
  time-dependent drift.
\newblock {\em \textit{Stoch. Proc. Appl.}}, 119:3238--3252, 2009.

\bibitem{MagdziarzOW}
M.~Magdziarz, S.~Orzel, and A.~Weron.
\newblock Option pricing in subdiffusive bachelier model.
\newblock {\em \textit{J. Stat. Phys.}}, 145:187--203, 2011.

\bibitem{MagdziarzSchilling}
M.~Magdziarz and R.~L. Schilling.
\newblock Asymptotic properties of brownian motion delayed by inverse
  subordinators.
\newblock {\em \textit{Proc. Amer. Math. Soc.}}, 143:4485--4501, 2015.

\bibitem{MagdziarzZorawik}
M.~Magdziarz and T.~Zorawik.
\newblock Stochastic representation of fractional subdiffusion equation. the
  case of infinitely divisible waiting times, l\'evy noise and
  space-time-dependent coefficients.
\newblock {\em To appear in \textit{Proc. Amer. Math. Soc.}}, 2015.
\newblock \texttt{ar{X}iv:1509.09051 [math.{PR}]}.

\bibitem{MS_1}
M.~M. Meerschaert and H-P. Scheffler.
\newblock Limit theorems for continuous-time random walks with infinite mean
  waiting times.
\newblock {\em \textit{J. Appl. Probab.}}, 41:623--638, 2004.

\bibitem{MS_2}
M.~M. Meerschaert and H-P. Scheffler.
\newblock Triangular array limits for continuous time random walks.
\newblock {\em \textit{Stoch. Proc. Appl.}}, 118:1606--1633, 2008.

\bibitem{MetzlerKlafter00}
R.~Metzler and J.~Klafter.
\newblock The random walk's guide to anomalous diffusion: a fractional dynamics
  approach.
\newblock {\em \textit{Phys. Rep.}}, 339(1):1--77, 2000.

\bibitem{Mitrinovic}
D.~S. Mitrinovi{\'c}, J.~E. Pe{\v{c}}ari{\'c}, and A.~M. Fink.
\newblock {\em \textit{Inequalities Involving Functions and Their Integrals and
  Derivatives}}.
\newblock Kluwer Academic Publishers, 1991.

\bibitem{Protter}
P.~Protter.
\newblock {\em \textit{Stochastic Integration and Differential Equations}}.
\newblock Springer, second edition, 2004.

\bibitem{Sato}
K-i. Sato.
\newblock {\em \textit{L{\'e}vy Processes and Infinitely Divisible
  Distributions}}.
\newblock Cambridge University Press, 1999.

\bibitem{Saxton}
M.~J. Saxton and K.~Jacobson.
\newblock Single-particle tracking: applications to membrane dynamics.
\newblock {\em \textit{Annu. Rev. Biophys. Biomol. Struct.}}, 26:373--399,
  1997.

\bibitem{ScalasViles}
E.~Scalas and N.~Viles.
\newblock A functional limit theorem for stochastic integrals driven by a
  time-changed symmetric $\alpha$-stable {L}{\'e}vy process.
\newblock {\em \textit{Stoch. Proc. Appl.}}, 124(1):385--410, 2014.

\bibitem{Steele}
M.~J. Steele.
\newblock {\em \textit{Stochastic Calculus and Financial Applications}}.
\newblock Springer, 2001.

\bibitem{Zaslavsky}
G.~M. Zaslavsky.
\newblock Fractional kinetic equation for {H}amiltonian chaos. chaotic
  advection, tracer dynamics and turbulent dispersion.
\newblock {\em \textit{Phys.\ D}}, 76:110--122, 1994.

\end{thebibliography}

\end{document}